\documentclass[reqno]{amsart}

\usepackage{amsmath}
\usepackage{amssymb}
\usepackage{amsthm}
\usepackage{enumerate}
\usepackage{mathrsfs} 
\usepackage{eqlist}
\usepackage{array}

\theoremstyle{plain}

\newtheorem{lemma}{Lemma}[section]
\newtheorem{proposition}{Proposition}[section]
\newtheorem{example}{Example}[section]
\newtheorem{corollary}{Corollary}[section]
\newtheorem{definition}{Definition}[section]

\numberwithin{equation}{section}

\begin{document}
	
	\title[Binomial coefficients and multifactorial numbers through generative grammars]{Binomial coefficients and multifactorial numbers through generative grammars}
	\author[Juan Gabriel Triana \and Rodrigo De Castro Korgi]%
	{Juan Gabriel Triana* \and Rodrigo De Castro Korgi** }
	
	\newcommand{\acr}{\newline\indent}
	
	\address{\llap{*\,}Departamento de Matem\'aticas\acr
		Universidad Nacional de  Colombia\acr
		Carrera 45 \# 26-85\acr
		Bogot\'a}
	\email{jtrianal@unal.edu.co}
	
	\address{\llap{**\,}Departamento de Matem\'aticas\acr
		Universidad Nacional de  Colombia\acr
		Carrera 45 \# 26-85\acr
		Bogot\'a}
	\email{rdecastrok@unal.edu.co}

	\keywords{multifactorial numbers, binomial coefficients, context-free grammars, formal derivative operator}
	
	\maketitle
	
\begin{abstract}
	In this paper, the formal derivative operator defined with respect to context-free grammars is used to prove some properties about binomial coefficients and multifactorial numbers. In addition, we extend the formal derivative operator to matrix grammars and show that multifactorial numbers can also be generated. 
\end{abstract}
\section{Introduction}
Let $\Sigma$ be an alphabet, whose letters are regarded as independent commutative indeterminates. Following \cite{Chen93}, a formal function over $\Sigma$ is defined recursively as follows:

\begin{enumerate}
	\item Every letter in $\Sigma$ is a formal function.
	\item If $u,v$ are formal functions, then $u+v$ and $uv$ are formal functions.
	\item If $f(x)$ is an analytic function, and $u$ is a formal function, then $f(u)$ is a formal function.
	\item Every formal function is constructed as above in a finite number of steps. 
\end{enumerate}

A context free grammar $G$ over $\Sigma$ is defined as a set of substitution rules replacing a letter in $\Sigma$ by a formal function over $\Sigma$.  

\begin{definition}\label{operador}
	Given a context-free grammar $G$ over $\Sigma$, the formal derivative operator $D$, with respect to $G$, is defined in the following way:
	
	\begin{enumerate}
		\item For $u,v$ formal functions, $D(u+v)=D(u)+D(v) \text{ and  }D(uv)=D(u)v+uD(v).$ \label{prop1}
		\item If $f(x)$ is an analytic function and $u$ is a formal function,  $D(f(u))=\dfrac{\partial f(u)}{\partial u}D(u)$. 
		\item For $a\in \Sigma$, if $a\rightarrow w$ is a production in $G$, $w$ being   a formal function, then $D(a)=w$; in other cases  $a$ is called a constant and  $D(a)=0$.
	\end{enumerate}
\end{definition}
We next define the iteration of the formal derivative operator. 
\begin{definition}\label{opplus1}
	For a formal function $u$, we define $D^{n+1}(u)=D(D^n(u))$ for $n\geq 0$, and $D^0(u)=u$.
\end{definition}

For instance, given the context-free grammar $G=\{a \rightarrow a+b ; b \rightarrow b\}$, we have $D^0(a)=a$, $D(a)=a+b$, $D(b)=b$, $D(ab)=D(a)b+aD(b)=(a+b)b+ab = b^2+2ab$, and $D^2(a)=D(D(a))$ so $D^2(a)=D(a+b)=D(a)+D(b)=a+2b$.

\medskip

In \cite{Chen93}, formal functions and the formal derivative operator were used to study formal power series in combinatorics. In addition, the formal derivative  operator, defined with respect to context-free grammars, has been used to study increasing trees \cite{ChenFu}, triangular arrays \cite{triangular}, permutations \cite{MaEnumeration}, Stirling permutations of the second kind \cite{MaNan}, and for generating some combinatorial numbers such as Whitney numbers \cite{Lotka}, Ramanujan's numbers \cite{Ramanujan}, Eulerian numbers \cite{ChenFu}, Stirling numbers \cite{MaStirling}, among others \cite{Ma2}. In the same way, some families of polynomials such as Bessel polynomials \cite{Ma2}, Eulerian polynomials \cite{Ma2018}, and other types of polynomials \cite{Dumont}, have been studied by grammatical methods.

\medskip

In this paper, the formal derivative operator defined with respect to context-free grammars is used to prove some properties about binomial coefficients and multifactorial numbers. In addition, we extend the formal derivative operator to matrix grammars and show that multifactorial numbers can also be generated.

\medskip

Our purpose is to obtain properties by means of the formal derivative operator; consequently most proofs are carried out by induction rather than by combinatorial arguments.

\section{Some properties of the formal derivative operator}

The formal derivative operator of Definition \ref{operador} preserves many of the properties of the differential operator in elementary calculus. In the following propositions we state and prove some of them.

\begin{proposition}\label{Dav}
	If $v$ is a formal function and $\alpha \in \mathbb{R}$, then $ D(\alpha v)=\alpha D(v)$.
\end{proposition}
\begin{proof}
	Let $f(x)=\alpha x$. Since $f(x)$ is an analytic function and $v$ is a formal function, from definition \ref{operador} we get $D(f(v)) =\dfrac{\partial f(v)}{\partial v} D(v)= \alpha D(v)$.
\end{proof}

Since $u$ and $v$ are formal functions and $\alpha \in \mathbb{R}$, we have  $D(u+v)=D(u)+D(v)$, by definition \ref{operador}, and $D(\alpha v)=\alpha D(v)$, by proposition \ref{Dav}, thus $D$ is a linear operator. 

\begin{proposition}\label{Dv1}
	If $v$ is a formal function, then $ D(v^{n})=nv^{n-1}D(v)$  for $n\in \mathbb{Z}$.
\end{proposition}
\begin{proof}
	Let $f(x)=x^{n}$. Since $f(x)$ is an analytic function and $v$ is a formal function, by definition \ref{operador}, we have $D(f(v))=\dfrac{\partial f(v)}{\partial v} D(v)= nv^{n-1} D(v)$. 
\end{proof}

By Proposition \ref{Dv1},  for each context-free grammar with a production $a\rightarrow w$ we have $D(a^0)=0$.
\begin{proposition}[Quotient's rule]\label{CocienteD}
	If $u,v$ are formal functions, then $ D(uv^{-1})=(D(u)v-uD(v))v^{-2}$.
\end{proposition}
\begin{proof}
	Let $u,v$ be formal functions; by definition \ref{operador}, $D(uv^{-1})=D(u)v^{-1} + uD(v^{-1})$. By proposition \ref{Dv1}, $D(v^{-1})=-v^{-2}D(v)$. Hence $D(uv^{-1})= D(u)v^{-1}-uv^{-2}D(v)= (D(u)v-uD(v))v^{-2}$. 
\end{proof}
The generalized product rule as well as Leibniz's formula from calculus also hold for formal functions cf. \cite{Chen93}; detailed proofs (by induction) are given in \cite{Thesis}.
\begin{proposition}[Generalized product rule]\label{productos}
	If $u_1,u_2, \ldots , u_n$ are formal functions, then
	$$
	D(u_1u_2 \ldots u_n) = D(u_1)u_2\ldots u_n + D(u_2)u_1u_3\ldots u_n +\cdots + D(u_n)u_1u_2\ldots u_{n-1}.
	$$
\end{proposition}
\begin{proposition}[Leibniz's formula]\label{leibniz}
	If $u,v$ are formal functions, then 
	\[
	D^n(uv)=\sum\limits_{k=0}^n \binom{n}{k}D^k(u)D^{n-k}(v).
	\]
\end{proposition}
Given a context-free grammar, if $D(a)\neq D(b)$ then $D^n(a)\neq D^n(b)$ does not necessarily hold for $n\geq 2$. The grammar $G=\{ a \rightarrow ab \ ; \ b\rightarrow  ac \ ; \ c\rightarrow b^2+ac-bc \}$ provides a counterexample: 
\begin{align*}
D^2(a) &= D(D(a)) & & &  D^2(b) &= D(D(b)) & & \\ 
&= D(ab) & &     &  &= D(ac) & & \\ 
&= D(a)b+aD(b) & & & &=D(a)c+aD(c) \\
&= (ab)b+a(ac)  & & &    &=(ab)c+a(b^2+ac-bc) & &\\
&= ab^2+a^2c.    & &&  &= ab^2+a^2c. & &
\end{align*} 
In the above example it is clear that $D^{n}(a)=D^n(b)$ for $n\geq 3$. Actually, in general this is always the case: if $D^k(a)=D^k(b)$, for some $k$, then $D^n(a)=D^n(b)$ for all $n>k$. That is so because $n$ can be written as $n=m+k$, and we have $D^n(a)= D^m(D^k(a))= D^m(D^k(b)) = D^{m+k}(b)= D^{n}(b)$.

\medskip

On the other hand, from $D(a^2)=D(b^2)$ does not necessarily  follow that $D(a)=D(b)$. For instance, given the grammar $G=\{a\rightarrow ab ; b \rightarrow a^2 \}$,  $D(a^2)=2aD(a)=2a^2b$ and  $D(b^2)= 2bD(b)=2a^2b$; however $D(a)\neq D(b)$. Similarly, if  $D(a^2)=D(ab)$, then  $D(a)=D(b)$ does not necessarily hold; for instance, for the grammar  $G=\{ a\rightarrow ab  ; b \rightarrow 2ab-b^2 \}$ we have  $D(ab)= 2a^2b$ and $D(a^2)=2a^2b$; however $D(a)\neq D(b)$. These examples provide useful  insight and allow us to state the following assertion.
\begin{proposition}
	There is no context-free grammar such that $	D(a^2)=D(b^2)=D(ab)$, with $a\neq b$ and   $D(a), D(b) \neq 0$.
\end{proposition}
\begin{proof}
	Since $D(a^2)=D(ab)$ we get $2aD(a)= D(a)b+aD(b)$, thus obtaining
	\begin{equation}\label{sis1}
	(2a-b)D(a) - aD(b)=0.
	\end{equation} 
	Similarly, from $D(b^2)= D(ab)$ we have $2bD(b)= D(a)b+aD(b)$, so
	\begin{equation}\label{sis2}
	- bD(a)+(2b-a)D(b)=0.
	\end{equation}
	From (\ref{sis1}) and (\ref{sis2}) we obtain the following system of linear equations	
	\begin{equation}\label{system}
	\begin{bmatrix}
	2a-b & -a \\
	-b & -a+2b
	\end{bmatrix}
	\begin{bmatrix}
	D(a) \\
	D(b)
	\end{bmatrix}=
	\begin{bmatrix}
	0 \\
	0
	\end{bmatrix}.
	\end{equation}
	For the matrix	$A=\left(\begin{smallmatrix}
	2a-b & -a \\
	-b & -a+2b
	\end{smallmatrix}\right)$, 
	$\det(A) = -2a^2 + 4ab -2b^2=-2(a-b)^2$;
	since $a\neq b$, $\det(A)\neq 0$. But the system (\ref{system}) is homogeneous, therefore has a single unique solution $D(a)=D(b)=0$.
\end{proof}

\section{Binomial coefficients and multifactorial numbers through context-free grammars}\label{context-freenumbers}

In this section we consider two specific context-free grammars and use the formal derivative operator defined on them to prove some properties about binomial coefficients and multifactorial numbers.

\begin{lemma}\label{basicocombinado}
	If $G=\{ a \rightarrow a \}$, then $D^n(a^m)=m^na^m$ for all  $m,n\geq 0$.
\end{lemma}
\begin{proof}
	We argue by induction on $n$. By definition  \ref{opplus1}, $D^0(a^m)=a^m$. Assuming that $D^n(a^m)=m^na^m$, $D^{n+1}(a^m)= D(D^n(a^m))= D(m^na^m)= m^n(ma^{m-1}D(a))=m^{n+1}a^m$.
\end{proof}

The following proposition states two well known properties regarding binomial coefficients for which we give new proofs based on the context-free grammar of Lemma \ref{basicocombinado}. Standard combinatorial proofs are ready available, see for instance \cite{Brualdi}.
\begin{proposition}\label{binom1}
	For $n\in \mathbb{N}$: 
	\begin{enumerate}
		\item $\displaystyle{\sum\limits_{k=0}^n \binom{n}{k}} = 2^n$.
		\item	$\displaystyle{\sum\limits_{k=0}^n (-1)^k\binom{n}{k}} = 0$.
	\end{enumerate}
\end{proposition}
\begin{proof}
	Let $G$ be the grammar $\{a\rightarrow a\}$. 
	\begin{enumerate}
		\item  By Leibniz's formula,
		\begin{equation}\label{eq1}
		D^n(a^2)= \sum\limits_{k=0}^n \binom{n}{k} D^{n-k}(a)D^k(a). 
		\end{equation}
		By taking $m=1$ and $m=2$ in Lemma \ref{basicocombinado} we have $D^n(a)=a$ and $D^n(a^2)=2^na^2$; substituting in (\ref{eq1}):
		\begin{equation*}
		2^na^2 = \sum\limits_{k=0}^n \binom{n}{k} (a)(a).
		\end{equation*}
		By equating the coefficient of $a^2$ it follows  $\sum\limits_{k=0}^n \binom{n}{k} = 2^n$.
		\item By Leibniz's formula,
		\begin{equation}\label{eq2}
		D^n(aa^{-1}) = \sum\limits_{k=0}^n \binom{n}{k} D^{n-k}(a)D^k(a^{-1}). 
		\end{equation}
		By taking $m=1$ and $m=-1$ in Lemma \ref{basicocombinado} we have $D^n(a)=a$ and $D^n(a^{-1})=(-1)^na^{-1}$; substituting in (\ref{eq2}):
		\begin{equation*}
		D^n(a^0)  = \sum\limits_{k=0}^n \binom{n}{k} (a)((-1)^ka^{-1}).  
		\end{equation*}
		Since $a^0=1$, $D^n(a^0)=0$. Thus  $\sum\limits_{k=0}^n (-1)^k \binom{n}{k} = 0$.
	\end{enumerate}
\end{proof}

The multifactorial numbers $n!_r$, also known as $(n,r)$-factorial numbers \cite{nkfactorials}, are given by the recurrence relation,
\[
n!_r=n(n-r)!_r \ \ \text{with} \ (1-r)!_r=\cdots  =(-1)!_r=0!_r=1.
\]
For instance, $(17)!_5=(17)(12)(7)(2)=2856$. When $r=1$ we get factorial numbers i.e., $n!_1=n!$;  when $r=2$ we get double factorial numbers i.e., $n!_2=n!!$, \cite{enciclopedia}. The following lemma establishes a connection between the context-free grammar $G=\{a\rightarrow a^{r+1}\}$ and multifactorial numbers.
\begin{lemma}\label{fact-Dnam}
	If $G=\{ a \rightarrow a^{r+1}  \}$, then $D^n(a^m)=\frac{(m+(n-1)r)!_r}{(m-r)!_r}a^{m+nr}$ for $n\geq 0$.
\end{lemma}
\begin{proof}
	We argue by induction on $n$.	Since $D^0(a^m)=a^m$, the lemma is true for $n=0$. Assuming that $D^n(a^m)=\frac{(m+(n-1)r)!_r}{(m-r)!_r}a^{m+n}$, $D^{n+1}(a^m)$	 is calculated  as follows:
	\begin{align*}
	D^{n+1}(a^m)&=D(D^{n}(a^m))\\
	&=D\left(\frac{(m+(n-1)r)!_r}{(m-r)!_r}a^{m+nr}\right) \\
	&=\frac{(m+(n-1)r)!_r}{(m-r)!_r}(m+nr)a^{m+nr-1}D(a) \\
	&=\frac{(m+nr)!_r}{(m-r)!_r}a^{m+nr-1}(a^{r+1}) \\	
	&=\frac{(m+nr)!_r}{(m-r)!_r}a^{m+(n+1)r}.
	\end{align*}
\end{proof}
\begin{proposition}\label{newfactorial}
	For $m,n,r$ integers, $m,n\geq 0$ and $r\geq 1$, we have
	\[
	\dfrac{(2m+(n-1)r)!_r}{(2m-r)!_r}= \displaystyle{\sum\limits_{k=0}^n \binom{n}{k}} \dfrac{(m+(k-1)r)!_r}{(m-r)!_r}\dfrac{(m+(n-k-1)r)!_r}{(m-r)!_r}.
	\]
\end{proposition}
\begin{proof}
	Let $G$ be the grammar $\{ a \rightarrow a^{r+1} \}$. By Leibniz's formula,
	\begin{equation}\label{eqproblem}
	D^n(a^{2m})= \sum\limits_{k=0}^n \binom{n}{k} D^k(a^m)D^{n-k}(a^m).
	\end{equation}
	By Lemma \ref{fact-Dnam}, $D^n(a^{2m})=\dfrac{(2m+(n-1)r)!_r}{(2m-r)!_r}a^{2m+nr}$; substituting in (\ref{eqproblem}) we get
	\begin{equation}\label{eqproblem2}
	\dfrac{(2m+(n-1)r)!_r}{(2m-r)!_r}a^{2m+nr}= \sum\limits_{k=0}^n \binom{n}{k} D^k(a^m)D^{n-k}(a^m).
	\end{equation}
	On the other hand, by taking $k$ and $n-k$ instead of $n$, respectively, in Lemma \ref{fact-Dnam}, we have $D^k(a^m)=\dfrac{(m+(k-1)r)!_r}{(m-r)!_r}a^{m+kr}$  and $D^{n-k}(a^m)=\dfrac{(m+(n-k-1)r)!_r}{(m-r)!_r}a^{m+(n-k)r}$. Substituting in the right hand side of (\ref{eqproblem2}):
	\begin{align*}
	&\phantom{xx} \sum\limits_{k=0}^n \binom{n}{k} \left[\dfrac{(m+(k-1)r)!_r}{(m-r)!_r}a^{m+kr} \right] \left[\dfrac{(m+(n-k-1)r)!_r}{(m-r)!_r}a^{m+(n-k)r} \right] \\
	&= \sum\limits_{k=0}^n \binom{n}{k} \dfrac{(m+(k-1)r)!_r}{(m-r)!_r}  \dfrac{(m+(n-k-1)r)!_r}{(m-r)!_r}a^{2m+nr}. 
	\end{align*}
	By equating the  coefficients of $a^{2m+nr}$ we obtain 
	\begin{equation*}
	\dfrac{(2m+(n-1)r)!_r}{(2m-r)!_r}= \displaystyle{\sum\limits_{k=0}^n \binom{n}{k}} \dfrac{(m+(k-1)r)!_r}{(m-r)!_r} \dfrac{(m+(n-k-1)r)!_r}{(m-r)!_r}.\qedhere
	\end{equation*}
\end{proof}

\begin{corollary}
	For $n\geq 0$ and $r \geq 1$, 	$((n+1)r)!_r=r\displaystyle{\sum\limits_{k=0}^n \binom{n}{k}} (kr)!_r ((n-k)r)!_r$.
\end{corollary}
\begin{proof}
	By taking $r=m$ in Proposition \ref{newfactorial}.
\end{proof}
\begin{corollary}\label{Doblefactpropiedad}
	For $n\geq 0$, 	$(2n)!! = \displaystyle{\sum\limits_{k=0}^n} \binom{n}{k}(2(n-k)-1)!!(2k-1)!!$.
\end{corollary}
\begin{proof}
	By taking $r=2$ and $m=1$ in Proposition \ref{newfactorial}, 
\end{proof}

A combinatorial proof of Corollary \ref{Doblefactpropiedad} can be found in \cite{fun}, Theorem 3.

\section{Formal derivative operator with respect to matrix grammars}
Matrix grammars were introduced in \cite{matrix}  as a generalization of standard context-free grammars. Following \cite{india}, a matrix grammar is defined  as follows.
\begin{definition}
	A matrix grammar $G$ is a quadruple $G=(V_N,V_T,S,M)$, such that:
	\begin{enumerate}
		\item $V_N$ is a finite set of objects, called variables.
		\item $V_T$ is a finite set of objects, called terminal symbols.
		\item $S \in V_N$ is a symbol called the start variable.
		\item $M$ is a finite set of sequences of the form $[A_1 \rightarrow x_1 , \ldots , A_n \rightarrow x_n ]$,  with $A_i \in (V_N\cup V_T)^+$ and $x_i \in (V_N\cup V_T)^*$, $i=1,\ldots , n$.
	\end{enumerate}
\end{definition}
$G$ is a matrix grammar of type $k$ if and only if each $[A_1 \rightarrow x_1 , \ldots ,A_n \rightarrow x_n ]$ in $M$ is a grammar of type $k$ in the Chomsky hierarchy \cite{india}. For instance, the grammars of type $2$ $g_1=\{ a\rightarrow a+b \ ; \ b \rightarrow b \}$ and $ g_2= \{ a\rightarrow a \ ; \ b \rightarrow a-b \}$ can be unified in the matrix grammar of type $2$, $G=\left\{ [ a\rightarrow a+b \ ; \ b \rightarrow b ]\ ,\ [ a\rightarrow a \ ; \ b \rightarrow a-b ] \right\}$. Matrix grammars of type $2$ are named context-free matrix grammars in \cite{Paun-libro}, and we use that terminology henceforth.

\medskip

We now proceed to extend the definition of the formal derivative operator $D$ (Definition \ref{operador}) to context-free matrix grammars.

\begin{definition}\label{operadormatriz}
	Let $G=\{ g_1, \ldots , g_n \}$ be a context-free matrix grammar, i.e., each $g_i$ is a context-free grammar: $g_i = [ a_{1} \rightarrow w_{i1}\ ;  \ldots ; a_{m} \rightarrow w_{im}  ]$, each $w_{ik}$ being a formal function. If $a_{j} \rightarrow w_{ij}$ is a production in $g_i$, then $D_i(a_j) =w_{ij}$. In other cases, $D_i(a_j)=0$.	Moreover, $D_{ik}(a_j)=D_i(D_k(a_j))$, $D_{ik}^{n+1}(a_j)=D_{ik}(D_{ik}^{n}(a_j))$ and $D_{ik}^0(a_j)=a_j$ for any formal function  $a_j$. 
\end{definition}

For the grammar $G=\left\{ [ a\rightarrow a+b \ ; \ b \rightarrow b ]\ ,\ [ a\rightarrow a \ ; \ b \rightarrow a-b ] \right\}$ it is easy to check by using Definition \ref{operadormatriz}   that $D_1(a)=a+b$, $D_1(b)=b$, $D_2(a)=a$ and $D_2(b)=a-b$. The following example shows that $D_{ik}(w)$ does not necessarily agrees with $D_{ki}(w)$.
\begin{example}\label{ejemplo}
	Given $G=\left\{ [ a\rightarrow a+b \ ; \ b \rightarrow b ]\ ,\ [ a\rightarrow a \ ; \ b \rightarrow a-b ] \right\}$, we calculate  $D_{21}(a+b)$ and $D_{12}(a+b)$.
	
	\medskip
	
	Here $g_1=\{ a\rightarrow a+b \ ; \ b \rightarrow b \}$ and $ g_2= \{ a\rightarrow a \ ; \ b \rightarrow a-b \} $, hence
	\begin{align*}
	D_{12}(a+b) &= D_1(D_2(a+b)) & & &  D_{21}(a+b) &= D_2(D_1(a+b)) & &  \\
	&= D_1(D_2(a)+D_2(b)) & & &   &= D_2(D_1(a)+D_1(b)) & &  \\
	&= D_1((a)+(a-b)) & &     &  &= D_2((a+b)+(b)) & &  \\ 
	&= D_1(2a-b) & & & &=D_2(a+2b) \\
	&= 2D_1(a)-D_1(b)  & & &    &= D_2(a)+2D_2(b) & & \\
	&= 2(a+b)-(b)  & & &    &= (a)+2(a-b) & & \\
	&= 2a+b.  & & &    &= 3a-2b. & & 
	\end{align*} 
\end{example}

As the following results show, we can also use matrix grammars for generating multifactorial numbers. For this purpose we need the following lemma.

\begin{lemma}\label{lemamatrix}
	For the grammar $G=\{ [a\rightarrow a ; b \rightarrow b] , [a\rightarrow a^rb ; b\rightarrow a^{r-1}b^2] \}$, we have
	\begin{enumerate}
		\item $D_1(a^mb^n)=(m+n)a^mb^n$.
		\item $D_2(a^mb^n)=(m+n)a^{m+r-1}b^{n+1}$.   	
	\end{enumerate}    	
\end{lemma}
\begin{proof}
	Since $g_1=\{a\rightarrow a ; b \rightarrow b\}$, we have
	\begin{align*}
	D_1(a^mb^n)&=ma^{m-1}b^nD_1(a)+na^mb^{n-1}D_1(b)\\
	&=ma^{m-1}b^n(a)+na^mb^{n-1}(b)\\
	&=ma^{m}b^n+na^mb^{n}.
	\end{align*}
	On the other hand, since $g_2=\{ a\rightarrow a^rb ; b\rightarrow a^{r-1}b^2 \}$ we get:  	
	\begin{align*}
	D_2(a^mb^n)&=ma^{m-1}b^nD_2(a)+na^mb^{n-1}D_2(b)\\
	&=ma^{m-1}b^n(a^rb)+na^mb^{n-1}(a^{r-1}b^2)\\
	&=ma^{m+r-1}b^{n+1}+na^{m+r-1}b^{n+1}.
	\end{align*}
	Hence $D_1(a^mb^n)=(m+n)a^mb^n$ and $D_2(a^mb^n)=(m+n)a^{m+r-1}b^{n+1}$.   	
\end{proof}

\begin{proposition}
	If $G=\{ [a\rightarrow a ; b \rightarrow b] , [a\rightarrow a^rb ; b\rightarrow a^{r-1}b^2] \}$, then for all $n\geq 0$:
	\begin{enumerate}
		\item $D^n_{12}(a)=(nr+1)!_r((n-1)r+1)!_ra^{nr-(n-1)}b^n$.
		\item $D^n_{21}(a)=((n-1)r+1)!_r^2a^{nr-(n-1)}b^n$.		
		\item $D^n_{12}(b)=(nr+1)!_r((n-1)r+1)!_ra^{nr-n}b^{n+1}$.
		\item $D^n_{21}(b)=((n-1)r+1)!^2_ra^{nr-n}b^{n+1}$.		
	\end{enumerate} 
\end{proposition}
\begin{proof}
	Here we prove (1); the other results can be proved similarly.
	
	\medskip
	
	We argue by induction on $n$. Since $D^0_{12}(a)=a$, the proposition is true for $n=0$. Assuming that $D^n_{12}(a)=(nr+1)!_r((n-1)r+1)!_ra^{nr-(n-1)}b^n$, $D^{n+1}_{12}(a)=D_{12}(D^n_{12}(a))$ can be expressed as
	\begin{equation}\label{eq1m}
	D^{n+1}_{12}(a)=(nr+1)!_r((n-1)r+1)!_rD_{1}(D_2( a^{nr-(n-1)}b^n) ).
	\end{equation}
	By Lemma \ref{lemamatrix}, $D_2( a^{nr-(n-1)}b^n)=(nr+1)a^{(n+1)r-n}b^{n+1}$. Substituting in (\ref{eq1m}),
	\begin{equation}\label{eq2m}
	D^{n+1}_{12}(a)=(nr+1)!_r((n-1)r+1)!_r(nr+1)D_{1}(a^{(n+1)r-n}b^{n+1}).
	\end{equation}
	By Lemma \ref{lemamatrix}, $D_1(a^{(n+1)r-n}b^{n+1})=((n+1)r+1)a^{(n+1)r-n}b^{n+1}$. Substituting in (\ref{eq2m}), and using the identities $((n+1)r+1)!_r=((n+1)r+1)(nr+1)!_r$ and $(nr+1)!_r=(nr+1)((n-1)r+1)!_r$ we conclude
	\begin{align*}
	D^{n+1}_{12}(a)&=(nr+1)!_r((n-1)r+1)!_r(nr+1)((n+1)r+1)a^{(n+1)r-n}b^{n+1}\\
	&=((n+1)r+1)!_r(nr+1)!_ra^{(n+1)r-n}b^{n+1}.\qedhere
	\end{align*}
\end{proof}


\end{document}